\documentclass{amsart}
\usepackage{graphicx}
\usepackage{ulem}
\usepackage{bm}
\usepackage{color}
\usepackage{amssymb, amsmath, amsthm}
\usepackage{bm}
\theoremstyle{plain}
\usepackage{graphicx}
\usepackage[aboveskip=-5pt,position=bottom]{caption}
\usepackage{natbib}
\usepackage{url}
\usepackage{enumerate}
\usepackage{rotating}

\newtheorem{lemma}{Lemma}[section]

\newtheorem{theorem}[lemma]{Theorem}

\newtheorem{prop}[lemma]{Proposition}
\newtheorem{exam}[lemma]{\normalfont \scshape
 Example}
\newtheorem{rem}[lemma]{\normalfont \scshape Remark}

\newcommand{\R}{\mathbb{R}}
\newcommand{\N}{\mathbb{N}}

\newcommand{\norm}[1]{\left\Vert#1\right\Vert}
\newcommand{\abs}[1]{\left\vert#1\right\vert}
\newcommand{\set}[1]{\left\{#1\right\}}

\newcommand{\eps}{\varepsilon}

\newcommand{\bfx}{\bm{x}}
\newcommand{\bfzero}{\bm{0}}

\newcommand{\bfy}{\bm{y}}
\newcommand{\bfZ}{\bm{Z}}

\newcommand{\bfeta}{\bm{\eta}}

\newcommand{\bftheta}{\bm{\vartheta}}

\DeclareMathOperator{\Var}{Var}
\DeclareMathOperator{\MSE}{MSE}
\DeclareMathOperator{\IMSE}{IMSE}

\usepackage[doublespacing]{setspace}

\allowdisplaybreaks[4]
\begin{document}

\title[On generalized max-linear models in max-stable random fields]{On generalized max-linear models in max-stable random fields}%

\author{Michael Falk, Maximilian Zott}
\address{University of Wurzburg,
Institute of Mathematics,  Emil-Fischer-Str. 30, 97074 W\"{u}rzburg, Germany.}
\email{michael.falk@uni-wuerzburg.de, maximilian.zott@uni-wuerzburg.de}


\begin{abstract}
In practice, it is not possible to observe a whole max-stable random field. Therefore, a way how to reconstruct a max-stable random field in $C\left([0,1]^k\right)$ by interpolating its realizations at finitely many points is proposed. The resulting interpolating process is again a max-stable random field. This approach uses a
\emph{generalized max-linear model}. Promising results have been established in the case $k=1$ in a previous paper. However, the extension to higher dimensions is not straightforward since we lose the natural order of the index space.
\end{abstract}

\maketitle
\section{Introduction and Preliminaries}

\citet{domeyri12} derive an algorithm to sample from the regular conditional distribution of a max-stable random field $\bfeta$, say, given the marginal observations $\eta_{s_1}=z_1,\dots,\eta_{s_k}=z_k$ for some $z_1,\dots,z_d$ from the state space and $k$ locations $s_1,\dots,s_d$. This, clearly, concerns the \emph{distribution} of $\bfeta$ and derived distributional parameters.

    Different to that, we try to \emph{reconstruct} $\bfeta$ from the observations $\eta_{s_1},\dots,\eta_{s_k}$. This is done by a \emph{generalized max-linear model} in such a way, that the interpolating process $\hat\bfeta$ is again a (standard) max-stable random field.

    As our approach is deterministic, once the  observations $\eta_{s_1}=z_1,\dots,\eta_{s_k}=z_k$ are given, a proper way to measure the performance of our approach is the \emph{mean squared error} (MSE). Convergence of the pointwise MSE as well as of the integrated MSE (IMSE) is established if the set of grid points $s_1,\dots,s_d$ gets dense in the index space.

A \emph{max-stable random process} with index set $T$ is a family of random variables $\bm\xi=(\xi_t)_{t\in T}$ with the property that there are functions $a_n:T\to\R^+_0$ and $b_n:T\to\R$, $n\in\N$, such that
\[
\left(\max_{i=1,\dotsc,n}\left(\frac{\xi^{(i)}_t-b_n(t)}{a_n(t)}\right)\right)_{t\in T}=_d\bm\xi,
\]
where $\bm\xi^{(i)}=(\xi^{(i)}_t)_{t\in T}$, $i=1,\dotsc,n$, are independent copies of $\bm\xi$ and '$=_d$' denotes equality in distribution. We get a max-stable random vector (rv) on $\R^d$ by putting $T=\{1,\dotsc,d\}$. Different to that, we obtain a max-stable process with continuous sample paths on some compact metric space $S$, if we set $T=S$ and require that the sample paths $\bm\xi(\omega):S\to\R$ realize in $C(S)=\{g\in\R^S:~g\text{ continuous}\}$, and that the norming functions $a_n,b_n$ are continuous as well. Max-stable random vectors, and processes, respectively, have been investigated intensely over the last decades. For detailed reviews of max-stable rv and processes, see for instance the monographies of \citet{beirgotese04}, \citet{dehaf06}, \citet{resn08}, \citet{fahure10} and \citet{davpari12} among others. Max-stable rv and processes are of enormous interest in extreme value theory since they are the only possible limit of linearly standardized maxima of independent and identically distributed rv or processes.

Clearly, the univariate margins of a max-stable random process are max-stable distributions on the real line. A max-stable random object $\bm\xi=(\xi_t)_{t\in T}$ is commonly called \emph{simple max-stable} in the literature if each univariate margin is unit Fr\'{e}chet distributed, i.\,e. $P(\xi_t\leq x)=\exp\left(-x^{-1}\right)$, $x>0$, $t\in T$. Different to that, we call a random process $\bm\eta=(\eta_t)_{t\in T}$ \emph{standard max-stable} if all univariate marginal distributions are standard negative exponential, i.\,e. $P(\eta_t\leq x)=\exp\left(x\right)$, $x\leq0$, $t\in T$. The transformation to simple/standard margins does not cause any problems, neither in the case of rv (see e.\,g. \citet{dehar77} or \citet{resn08}), nor in the case of rf with continuous sample paths (see e.\,g. \citet{ginhv90}).

It is well known (e.g. \citet{dehar77}, \citet{pick81}, \citet{fahure10}) that a rv $(\eta_1,\dotsc,\eta_d)$ is a \emph{standard max-stable rv} iff there exists a rv $(Z_1,\dotsc,Z_d)$ and some number $c\geq1$ with $Z_i\in[0,c]$ almost surely (a.\,s.) and $E(Z_i)=1$, $i=1,\dotsc,d$, such that for all $\bm x=(x_1,\dotsc,x_d)\leq\bm 0\in\R^d$
\[
P(\eta_1\leq x_1,\dotsc,\eta_d\leq x_d)=\exp\left(-\norm{\bm x}_D\right):=\exp\left(-E\left(\max_{i=1,\dotsc,d}\left(\abs{x_i}Z_i\right)\right)\right).
\]
The condition $Z_i\in[0,c]$ a.\,s. can be weakened to $P(Z_i\geq0)=1$. Note that $\norm\cdot_D$ defines a norm on $\R^{d}$, called \emph{$D$-norm}, with \emph{generator} $\bm Z$. The $D$ means dependence: We have independence of the margins of $\bm X$ iff $\norm\cdot_{D}$ equals the norm $\norm{\bm x}_1=\sum_{i=1}^d\abs{x_i}$, which is generated by $(Z_1,\dotsc,Z_d)$ being a random permutation of the vector $(d,0\dotsc,0)$. We have complete dependence of the margins of $\bm X$ iff $\norm\cdot_{D}$ is the maximum-norm $\norm{\bm x}_\infty=\max_{1\le i\le d}\abs{x_i}$, which is generated by the constant vector $(Z_1,\dotsc,Z_d)=(1,\dotsc,1)$. We refer to \citet[Section 4.4] {fahure10} for further details of  $D$-norms.

Let $S$ be a compact metric space. A standard max-stable process $\bm\eta=(\eta_t)_{t\in S}$ with sample paths in $\bar C^-(S):=\{g\in C(S):~g\leq 0\}$ is, in what follows, shortly called a \emph{standard max-stable process} (SMSP). Denote further by $E(S)$ the set of those bounded functions $f\in\R^S$ that have only a finite number of discontinuities and define $\bar E^-(S):=\{f\in E(S):~f\leq 0\}$. We know from \citet{ginhv90} that a process $\bm\eta=(\eta_t)_{t\in S}$ with sample paths in $C(S)$ is an SMSP iff there exists a stochastic process $\bm Z=(Z_t)_{t\in S}$ realizing in $\bar C^+(S):=\{g\in C(S):~g\geq 0\}$ and some $c\geq1$, such that $Z_t\leq c$ a.\,s., $E(Z_t)=1$, $t\in S$, and
\[
P(\bm\eta\leq f)=\exp\left(-\norm f_D\right):=\exp\left(-E\left(\sup_{t\in S}\left(\abs{f(t)}Z_t\right)\right)\right),\qquad f\in\bar E^-(S).
\]
Note that $\norm\cdot_D$ defines a norm on the function space $E(S)$, again called \emph{$D$-norm} with \emph{generator process} $\bm Z$. The functional $D$-norm is topologically equivalent to the sup-norm $\norm f_{\infty}=\sup_{t\in S}\abs{f(t)}$, which is itself a $D$-norm by putting $Z_t=1$, $t\in S$, see \citet{aulfaho11} for details.

At first it might seem unusual to consider the function space $E(S)$. The reason for that is that a suitable choice of the function $f\in\bar E^-(S)$ allows the incorporation of the finite dimensional marginal distributions by the relation $P(\bm\eta\leq f)=P(\eta_{t_i}\leq x_i,1\leq i\leq d)$.

The condition $P\left(\sup_{t\in S}Z_t\leq c\right)=1$ can be weakened to
\begin{equation}\label{eq:condition_generator}
E\left(\sup_{t\in S}Z_t\right)<\infty,
\end{equation}
see \citet[Corollary 9.4.5]{dehaf06}.

\section{Generalized max-linear models}\label{sec:model}

\subsection*{The model and some examples}

In this section we will approximate a given SMSP with sample paths in $\bar C^-\left([0,1]^k\right)$, where $k$ is some integer, by using a generalized max-linear model for the interpolation of a finite dimensional marginal distribution. The parameter space $[0,1]^k$ is chosen for convenience and could be replaced by any compact metric space $S$.

Let in what follows $\bm\eta=(\eta_{t})_{t\in [0,1]^k}$ be an SMSP with generator $\bm Z=(Z_{ t})_{ t\in [0,1]^k}$ and $D$-norm $\norm\cdot_{D}$. Choose pairwise different points $ s_1,\dotsc, s_d\in [0,1]^k$ and obtain a standard max-stable rv $(\eta_{ s_1},\dotsc,\eta_{ s_d})$ with generator $(Z_{ s_1},\dotsc,Z_{ s_d})$ and $D$-norm $\norm\cdot_{D_{1,\dotsc,d}}$, i.\,e.,
\[
P(\eta_{ s_1}\leq x_1,\dotsc,\eta_{ s_d}\leq x_d)=\exp\left(-E\left(\max_{i=1,\dotsc,d}\left(\abs{x_i}Z_{ s_i}\right)\right)\right)=:\exp\left(-\norm{\bm x}_{D_{1,\dotsc,d}}\right),
\]
$\bm x=(x_1,\dotsc,x_d)\leq\bm 0$. Our aim is to find another SMSP that interpolates the above rv.

Take functions $g_i\in\bar C^+\left([0,1]^k\right)$, $i=1,\dotsc,d$, with the property
\begin{equation}\label{eq:norming_functions_standardization}
\norm{(g_1( t),\dotsc,g_d( t))}_{D_{1,\dotsc,d}}=1\text{ for all } t\in[0,1]^k.
\end{equation}
Then the stochastic process $\hat{\bm\eta}=(\hat\eta_{ t})_{ t\in[0,1]^k}$ that is generated by the \emph{generalized max-linear model}
\begin{equation}\label{eq:generalized_max_linear_model}
\hat\eta_{ t}:=\max_{i=1,\dotsc,d}\frac{\eta_{ s_i}}{g_i( t)},\qquad  t\in[0,1]^k,
\end{equation}
defines an SMSP with generator
\begin{equation}\label{eq:generalized_max_linear_model_generator}
\hat Z_{ t}=\max_{i=1,\dotsc,d}\left(g_i( t)Z_{ s_i}\right),\qquad t\in[0,1]^k,
\end{equation}
due to property \eqref{eq:norming_functions_standardization}, see \citet{falhz13} for details. The case $\norm\cdot_{D_{1,\dotsc,d}}=\norm\cdot_1$ leads to the regular \emph{max-linear model}, cf. \citet{wansto11}.

If we want $\hat{\bm\eta}$ to interpolate $(\eta_{ s_1},\dotsc,\eta_{ s_d})$, then we only have to demand
\begin{equation}\label{eq:norming_functions_interpolation}
g_i( s_j)=\delta_{ij}:=\begin{cases}1,&\qquad i=j,\\0,&\qquad i\neq j,\end{cases}\quad 1\leq i,j\leq d.
\end{equation}
Recall that $\eta_{ s_i}$ is negative with probability one. We call $\hat{\bm\eta}$ the \emph{discretized version} of $\bm\eta$ with grid $\{ s_1,\dotsc, s_d\}$ and weight functions $g_1,\dotsc,g_d$, when the weight functions satisfy both \eqref{eq:norming_functions_standardization} and \eqref{eq:norming_functions_interpolation}.

\begin{exam}\label{exam:onedimensional_model}
\upshape
In the one-dimensional case $k=1$ the weight functions $g_i$ can be chosen as follows. Take a grid $0:=s_1<s_2<\cdots<s_{d-1}<s_d=:1$ of the interval $[0,1]$ and denote by $\norm\cdot_{D_{i-1,i}}$ the $D$-norm pertaining to $(\eta_{s_{i-1}},\eta_{s_i})$, $i=2,\dotsc,d$. Put
\begin{align*}
g_1(t)&:=\begin{cases}\dfrac{s_{2}-t}{\norm{(s_{2}-t,t)}_{D_{1,2}}},\quad &t\in[0,s_2], \\ 0,\quad &\text{else},\end{cases}\\
g_i(t)&:=\begin{cases}\dfrac{t-s_{i-1}}{\norm{(s_i-t,t-s_{i-1})}_{D_{i-1,i}}},\quad &t\in[s_{i-1},s_i], \\ \dfrac{s_{i+1}-t}{\norm{(s_{i+1}-t,t-s_{i})}_{D_{i,i+1}}},\quad &t\in[s_i,s_{i+1}], \\ 0,\quad &\text{else},\end{cases}\quad i=2,\dotsc,d-1,\\
g_d(t)&:=\begin{cases}\dfrac{t-s_{d-1}}{\norm{(s_d-t,t-s_{d-1})}_{D_{d-1,d}}},\quad &t\in[s_{d-1},1], \\ 0,\quad &\text{else}.\end{cases}
\end{align*}
This model has been studied intensely in \citet{falhz13}. The functions $g_1,\dotsc,g_d$ are continuous and satisfy conditions \eqref{eq:norming_functions_standardization} and \eqref{eq:norming_functions_interpolation}, so they provide an interpolating generalized max-linear model on $C[0,1]$.
\end{exam}

\begin{exam}\label{exam:multidimensional_model}
\upshape
Choose pairwise different points $s_1,\dotsc,s_d\in [0,1]^k$ and an arbitrary norm $\norm\cdot$ on $\R^k$. Define
\[
\tilde g_i(t):=\min_{j\neq i}\left(\norm{t-s_j}\right),\qquad t\in[0,1]^k,\; i=1,\dotsc,d.
\]
In order to normalize, put
\[
g_i(t):=\frac{\tilde g_i(t)}{\norm{(\tilde g_1(t),\dotsc,\tilde g_d(t))}_{D_{1,\dotsc,d}}},\quad t\in[0,1]^k,\quad i=1,\dotsc,d.
\]
These functions $g_i$ are well-defined since the denominator never vanishes: Suppose there is $ t\in[0,1]^k$ with $\tilde g_1( t)=\cdots=\tilde g_d( t)=0$. Then $\min_{j\neq i}\left(\norm{ t- s_j}\right)=0$ for all $i=1,\dotsc,d$. Now fix $i\in\{1,\dotsc,d\}$. There is $j\neq i$ with $ t= s_j$. But on the other hand, we have also $\min_{k\neq j}\left(\norm{ t- s_k}\right)=0$ which implies that there is $k\neq j$ with $ t= s_k= s_j$ which is a contradiction.

The functions $g_i$, $i=1,\dotsc,d$, are clearly functions in $\bar C^+\left([0,1]^k\right)$ that also satisfy condition \eqref{eq:norming_functions_standardization} and \eqref{eq:norming_functions_interpolation} as can be seen as follows. We have for $t\in[0,1]^k$
\begin{align*}
&\norm{\big(g_1(t),\dots,g_d(t)\big)}_{D_{1,\dots,d}}\\
&=\norm{\left(\frac{\tilde g_1(t)}{\norm{(\tilde g_1(t),\dots,\tilde g_d(t))}_{D_{1,\dots,d}}},\dots, \frac{\tilde g_d(t)}{\norm{(\tilde g_1(t),\dots,\tilde g_d(t))}_{D_{1,\dots,d}}} \right)}_{D_{1,\dots,d}}\\
&= \frac{\norm{\big(\tilde g_1(t),\dots,\tilde g_d(t)\big)}_{D_{1,\dots,d}}}{\norm{\big(\tilde g_1(t),\dots,\tilde g_d(t)\big)}_{D_{1,\dots,d}}}\\
&=1,
\end{align*}
which is condition \eqref{eq:norming_functions_standardization}. Note, moreover, that $\tilde g_i(s_j)=0$ if $i\not=j$. But this implies condition \eqref{eq:norming_functions_interpolation}:
\begin{align*}
g_i(s_j)&= \frac{\tilde g_i(s_j)}{\norm{\big(\tilde g_1(s_j),\dots,\tilde g_d(s_j)\big)}_{D_{1,\dots,d}}}\\
&=\frac{\tilde g_i(s_j)}{\norm{\big(0,\dots,0,\tilde g_j(s_j),0,\dots,0\big)}_{D_{1,\dots,d}}}\\
&=\frac{\tilde g_i(s_j)}{\tilde g_j(s_j) \norm{\big(0,\dots,0,1,0,\dots,0\big)}_{D_{1,\dots,d}}}\\
&= \frac{\tilde g_i(s_j)}{\tilde g_j(s_j)}=\delta_{ij}
\end{align*}
by the fact that a $D$-norm of each unit vector in $\R^d$ is one.
 Thus, we have found an interpolating generalized max-linear model on $C\left([0,1]^k\right)$.
\end{exam}

\subsection*{The mean squared error of the discretized version}

We start this section with a result that applies to bivariate standard max-stable rv in general.

\begin{lemma}\label{lem:properties_bivariate_smsrv}
Let $(X_1,X_2)$ be standard max-stable with generator $(Z_1,Z_2)$ and $D$-norm $\norm\cdot_D$.
\begin{enumerate}[(i)]
\item \[E(X_1X_2)=\int_0^{\infty}\frac{1}{\norm{(1,u)}^2_{D}}~du.\]
\item \[E(|Z_1-Z_2|)=2\left(\norm{(1,1)}_{D}-1\right).\]
\end{enumerate}
\end{lemma}

\begin{proof}
\begin{enumerate}[(i)]
\item See \citet[Lemma 3.6]{falhz13}.
\item The assertion follows from the general identity $\max(a,b)=\frac12(a+b+\abs{a-b})$.
\end{enumerate}
\end{proof}

Let $\hat{\bm\eta}=(\hat\eta_{ t})_{ t\in [0,1]^k}$ be the discretized version of $\bm\eta=(\eta_{ t})_{ t\in[0,1]^k}$ with grid $\{ s_1,\dotsc, s_d\}$ and weight functions $g_1,\dotsc,g_d$. In order to calculate the mean squared error of $\hat\eta_t$, we need the following lemma.

\begin{lemma}\label{lem:eta_hateta_sms}
Let $\hat{\bm Z}=(\hat Z_t)_{t\in[0,1]^k}$ be the generator of $\hat{\bm\eta}$ that is defined in \eqref{eq:generalized_max_linear_model_generator}. For each $t\in[0,1]^k$, the rv $(\eta_{ t},\hat\eta_{ t})$ is standard max-stable with generator $(Z_t,\hat Z_t)$ and $D$-norm
\[
\norm{(x,y)}_{D_{ t}}=E\left(\max\left(\abs x Z_t,\abs y\hat Z_t\right)\right)=\norm{\left(x,g_1( t)y,\dotsc,g_d( t)y\right)}_{D_{ t, s_1,\dotsc, s_d}},
\]
where $\norm{\cdot}_{D_{ t, s_1,\dotsc, s_d}}$ is the $D$-norm pertaining to $(\eta_{ t},\eta_{ s_1},\dotsc,\eta_{ s_d})$.
\end{lemma}

\begin{proof}
As $\bm Z=(Z_{ t})_{ t\in[0,1]^k}$ is a generator of $\bm\eta$, we have for $x,y\leq 0$
\begin{align*}
P(\eta_{ t}\leq x,\hat\eta_{ t}\leq y)&=P(\eta_{ t}\leq x,\eta_{ s_1}\leq g_1( t)y,\dotsc,\eta_{ s_d}\leq g_d( t)y)\\
&=\exp\left(-E\left(\max\left(\abs xZ_{ t},\abs y\max\left(g_1( t)Z_{ s_1},\dotsc,g_d( t)Z_{ s_d}\right)\right)\right)\right)\\
&=\exp\left(-E\left(\max\left(\abs xZ_{ t},\abs y\hat Z_t\right)\right)\right).
\end{align*}
Then the assertion follows from the fact that $\hat Z_t\geq 0$ and $E(\hat Z_t)=1$.
\end{proof}

We can now use the preceding Lemmas to compute the mean squared error.

\begin{prop}\label{prop:mean squared error}
The mean squared error of $\hat{\eta_{ t}}$ is given by
\[
\MSE\left(\hat\eta_{t}\right):=E\left(\left(\eta_{ t}-\hat{\eta}_{ t}\right)^2\right)=2\left(2-\int_0^{\infty}\frac{1}{\norm{(1,u)}^2_{D_{ t}}}~du\right),\qquad t\in[0,1]^k.
\]
\end{prop}

\begin{proof}
Due to Lemma \ref{lem:eta_hateta_sms}, $(\eta_t,\hat\eta_t)$ is standard max-stable. Therefore, Lemma \ref{lem:properties_bivariate_smsrv} (i) and the fact that $E(\eta_t)=E(\hat\eta_t)=-1$ and $\Var(\eta_t)=\Var(\hat\eta_t)=1$ yield
\[
\MSE\left(\hat\eta_{ t}\right)=E\left(\eta_t^2\right)-2E\left(\eta_t\hat\eta_t\right)+E\left(\hat \eta_t^2\right)=4-2\int_0^{\infty}\frac{1}{\norm{(1,u)}^2_{D_{ t}}}~du.
\]
\end{proof}

\begin{lemma}\label{lem:mse_inequality}
The mean squared error of $\hat\eta_t$ satisfies
\[
\MSE\left(\hat\eta_t\right)\leq6 E\left(\abs{Z_t-\hat Z_t}\right),\qquad t\in[0,1]^k.
\]
\end{lemma}

\begin{proof}
We have
\begin{align*}
&2-\int_0^{\infty}\frac{1}{\norm{(1,u)}_{D_t}^2}~du\\
&=\int_0^{\infty}\frac{1}{\norm{(1,u)}_{\infty}^2}~du-\int_0^{\infty}\frac{1}{\norm{(1,u)}_{D_t}^2}~du\\
&=\int_0^{\infty}\left(\norm{(1,u)}_{D_t}-\norm{(1,u)}_{\infty}\right)\frac{\norm{(1,u)}_{D_t}+\norm{(1,u)}_{\infty}}{\norm{(1,u)}_{D_t}^2\norm{(1,u)}_{\infty}^2}~du\\
&=\int_0^{1}\left(\norm{(1,u)}_{D_t}-1\right)\frac{\norm{(1,u)}_{D_t}+1}{\norm{(1,u)}_{D_t}^2}~du+\int_1^{\infty}\left(\norm{(1,u)}_{D_t}-u\right)\frac{\norm{(1,u)}_{D_t}+u}{u^2\norm{(1,u)}_{D_t}^2}~du\\
&\leq 3\int_0^{1}\left(\norm{(1,u)}_{D_t}-1\right)~du+2\int_1^{\infty}\frac{\norm{(1/u,1)}_{D_t}-1}{u^2}~du\\
&=:3I_1+2I_2.
\end{align*}
Since every $D$-norm is monotone, we have
\[
\norm{(1,u)}_{D_t}\leq \norm{(1,1)}_{D_t},~ u\in[0,1],\text{ and } \norm{(1/u,1)}_{D_t}\leq \norm{(1,1)}_{D_t},~ u>1,
\]
and, thus, by Lemma \ref{lem:properties_bivariate_smsrv} (ii)
\begin{equation*}
I_1+I_2\leq \norm{(1,1)}_{D_t}-1+\left(\norm{(1,1)}_{D_t}-1\right)\int_1^{\infty}u^{-2}~du=E\left(\abs{Z_t-\hat Z_t}\right).
\end{equation*}
\end{proof}

\begin{rem}\upshape
The upper bound $E\left(\abs{Z_t-\hat Z_t}\right)$ in Lemma \ref{lem:mse_inequality} gets small if the distance between $t$ and its nearest neighbor $s_j$, say, in the grid $\set{s_1,\dots,s_d}$ gets small, which can be seen as follows. The triangle inequality implies
\[
\abs{Z_t-\hat Z_t} \le \abs{Z_t- Z_{s_j}} + \abs{Z_{s_j}-\max_{i=1,\dots,d}\left(g_i(t)Z_{s_i}\right)}.
\]
From the condition $g_i(s_j)=\delta_{ij}$ we obtain the representation
\[
Z_{s_j}= \max_{i=1,\dots,d}\left(g_i(s_j)Z_{s_i}\right)
\]
and, thus,
\begin{align*}
\abs{Z_{s_j}-\max_{i=1,\dots,d}\left(g_i(t)Z_{s_i}\right)}
&=\abs{ \max_{i=1,\dots,d}\left(g_i(s_j)Z_{s_i}\right) - \max_{i=1,\dots,d}\left(g_i(t)Z_{s_i}\right) }\\
&\le \max_{i=1,\dots,d}\left(\abs{g_i(t)-g_i(s_j)}Z_{s_i}\right)
\end{align*}
by elementary arguments. As a consequence we obtain
\begin{align*}
&E\left(\abs{Z_t-\hat Z_t}\right)\\
&\le E\left(\abs{Z_t-Z_{s_j}} \right) + E\left(\max_{i=1,\dots,d}\left(\abs{g_i(t)-g_i(s_j)}Z_{s_i}\right) \right)\\
&= E\left(\abs{Z_t-Z_{s_j}} \right) + \norm{\big(\abs{g_1(t)-g_1(s_j)},\dots, \abs{g_d(t)-g_d(s_j)}  \big)}_{D_{1,\dots,d}}\\
&\le E\left(\abs{Z_t-Z_{s_j}} \right) + \max_{i=1,\dots,d} \abs{g_i(t)-g_i(s_j)} \; \norm{(1,\dots,1)}_{D_{1,\dots,d}}\\
&\to_{\abs{t-s_j}\to 0}0
\end{align*}
by the fact that each $D$-norm $\norm\cdot_D$ is monotone, i.e., $\norm{\bfx}_D\le\norm{\bfy}_D$ if $\bfzero\le\bfx\le\bfy\in\R^d$, and by the continuity of the functions $g_1,\dots,g_d$ and $\bfZ$.
\end{rem}

\begin{exam}
\upshape
Choose as a generator process $\bfZ=(Z_t)_{t\in[0,1]^k}$ of a $D$-norm
\[
Z_{t}:=\exp\left(X_{t}-\frac{\sigma^2(t)}2 \right),\qquad t\in[0,1]^k,
\]
where $\left(X_{t}\right)_{t\in \R^k}$ is a continuous zero mean Gaussian process with stationary increments, $\sigma^2(t):= E\left(X_{t}^2\right)$ and $X_0=0$. This model was originally created by \citet{brore77}, and developed by \citet{kaschdeh09} for max-stable random fields $\bftheta=(\vartheta_{t})_{t\in[0,1]^k}$ with Gumbel margins, i.e., $P(\vartheta_t\le x)=\exp(-e^{-x})$, $x\in\R$. The transformation to a SMSP $(\eta_{t})_{t\in[0,1]^k}$ is straightforward by putting $\eta_{t}:=-\exp(-\vartheta_{t})$, $t\in[0,1]^k$.

Explicit formulae for the corresponding $D$-norm
\[
\norm{f}_D = E\left(\sup_{t\in[0,1]^k} (\abs{f(t)}Z_{t})\right),\qquad f\in E([0,1]^k),
\]
are only available for bivariate $\norm\cdot_{D_{t_1,t_2}}$ and trivariate $\norm\cdot_{D_{t_1,t_2,t_3}}$ $D$-norms pertaining to the random vectors $(\eta_{t_1},\eta_{t_2})$ and $(\eta_{t_1},\eta_{t_2},\eta_{t_3})$, respectively, see \citet{huserdav13}. In the bivariate case we have for $(x_1,x_2)\in\R^2$
\begin{align*}
\norm{(x_1,x_2)}_{D_{t_1,t_2}}
&= \abs{x_1} \Phi\left(\frac{\sigma(\abs{t_1-t_2})}2 + \frac1{\sigma(\abs{t_1-t_2})}
\log\left(\frac{\abs{x_1}}{\abs{x_2}}\right)\right)\\
&\hspace*{1cm}+ \abs{x_2} \Phi\left(\frac{\sigma(\abs{t_1-t_2})}2 + \frac1{\sigma(\abs{t_1-t_2})}
\log\left(\frac{\abs{x_2}}{\abs{x_1}}\right)\right),
\end{align*}
where $\Phi$ denotes the standard normal distribution function and the absolute value $\abs{t_1-t_2}$ is meant component wise, see \citet[Remark 24]{kabl09}.

This Brown-Resnick model could in particular  be used for the generalized max-linear model in dimension $k=1$ as in Example \ref{exam:onedimensional_model}, since in this case the approximation $\hat\bfeta$ of $\bfeta$ only uses bivariate $D$-norms $\norm\cdot_{t_1,t_2}$.

\end{exam}

\section{A generalized max-linear model based on kernels}

\subsection*{The model}

There is the need for the definition of $d$ functions $g_1,\dots,g_d$ satisfying certain constraints in the \emph{ordinary} generalized max-linear model with $d=d(n)$ tending to infinity as the grid $s_1,\dots,s_d$ gets dense in the index set. 
For the kernel approach introduced in this section, this is reduced to the choice of just one kernel and a bandwidth. And in this case we can establish convergence to zero of MSE and IMSE as the grid gets dance, essentially without further conditions. 
This approach was briefly mentioned in \citet{falhz13} and is evaluated here.

The disadvantages are: The interpolation is not an exact one at the grid points, i.e., $\hat\eta_{s_j}\not=\eta_{s_j}$. This is due to the fact that the generated functions do not satisfy the condition $g_i(s_j)=\delta_{ij}$ exactly, but only in the limit as $h$ tends to zero, see Lemma \ref{lem:convergence_to_kronecker}. The choice of an optimal bandwidth, which is statistical folklore in kernel density estimation, is still an open problem here.

 Again, throughout the whole section, let $\bm\eta=(\eta_t)_{t\in[0,1]^k}$ be an SMSP with generator $\bm Z=(Z_t)_{t\in[0,1]^k}$ and denote by $\norm\cdot_{s_1,\dotsc,s_d}$ the $D$-norm pertaining to $(\eta_{s_1},\dotsc,\eta_{s_d})$.

 Let $K:[0,\infty)\to[0,1]$ be a continuous and strictly monotonically decreasing function (kernel) with the two properties
\begin{equation}\label{eq:condition_on_kernel}
K(0)=1,\qquad \lim_{x\to\infty}\frac{K(ax)}{K(bx)}=0,\quad 0\le b< a.
\end{equation}
The exponential kernel $K_e(x)=\exp(-x)$, $x\ge 0$, is a typical example. Choose an arbitrary norm $\norm\cdot$ on $\R^k$ and a grid of pairwise different points $\{ s_1,\dotsc, s_d\}$ in $[0,1]^k$. Put for $i=1,\dotsc,d$ and the bandwidth $h>0$
\[
g_{i,h}( t):=\frac{K(\norm{t-s_i}/h)}{\norm{(K(\norm{t-s_1}/h),\dots,K(\norm{t-s_d}/h))}_{D_{s_1,\dotsc,s_d}}},\quad t\in[0,1]^k.
\]
Define for $i=1,\dots,d$
\begin{equation}\label{eq:set_closest_points}
N(s_i):=\set{t\in[0,1]^k:\,\norm{t-s_i}\le \norm{t-s_j},\,j\not=i},
\end{equation}
which is the set of those points $t\in[0,1]^k$ that are closest to the grid point $s_i$.

\begin{lemma}\label{lem:convergence_to_kronecker}
We have for arbitrary $t\in[0,1]^k$ and $1\le i\le d$
\[
g_{i,h}(t)\to_{h\downarrow 0}\begin{cases}
1&,\mbox{ if }t=s_i\\
0&,\mbox{ if }t\not\in N(s_i)
\end{cases}
\]
as well as $g_{i,h}(t)\le 1$.
\end{lemma}

\begin{proof}
The convergence $g_{i,h}( s_i)\to_{h\downarrow0}1$ follows from the fact that $K(0)=1$ and that the $D$-norm of a unit vector is 1.
The fact that an arbitrary $D$-norm is bounded below by the sup-norm together with the monotonicity of $K$ implies for $t\in[0,1]^k$
\begin{equation*}
g_{i,h}(t)\le \frac{K\left(\norm{t-s_i}/h\right)}{\max_{1\le j\le d}K\left(\norm{t-s_j}/h\right)}
=\frac{K\left(\frac{\norm{t-s_i}}h\right)}{K\left(\frac{\min_{1\le j\le d}\norm{t-s_j}}h\right)}\le 1.
\end{equation*}
Note that $K\left(\norm{t-s_i}/h\right)/K\left(\min_{1\le j\le d}\norm{t-s_j}/h\right) \to_{h\downarrow 0}0$
if $t\not\in N(s_i)$ by the required growth condition on the kernel $K$ in \eqref{eq:condition_on_kernel}.
\end{proof}

The above Lemma shows in particular $g_{i,h}( s_j)\to_{h\downarrow0}\delta_{ij}$ which is close to condition \eqref{eq:norming_functions_interpolation}. Obviously, the functions $g_{i,h}$ are constructed in such a way that condition \eqref{eq:norming_functions_standardization} holds exactly. Therefore, we obtain the generalized max-linear model
\begin{equation*}
\hat\eta_{ t,h}=\max_{i=1,\dotsc,d}\frac{\eta_{ s_i}}{g_{i,h}( t)},\qquad t\in[0,1]^k,
\end{equation*}
which does not interpolate $(\eta_{ s_1},\dotsc,\eta_{ s_d})$ exactly, but $\hat\eta_{s_i,h}$ converges to $\eta_{s_i}$ as $h\downarrow0$. Note that the limit functions $\lim_{h\downarrow0}g_{i,h}$ are not necessarily continuous: For instance, there may be $t_0\in[0,1]^k$ with $\norm{t_0-s_1}=\cdots=\norm{t_0-s_d}$. Then $ t_0\in\partial N( s_1)$ and $\lim_{h\downarrow0}g_{1,h}( t_0)=1/\norm{(1,\dotsc,1)}_{D_{1,\dotsc,d}}$, but $\lim_{h\downarrow0}g_{1,h}( t)=0$ for all $ t\notin N( s_1)$ due to Lemma \ref{lem:convergence_to_kronecker}.

\subsection*{Convergence of the mean squared error}

In this section we investigate a sequence of kernel-based generalized max-linear models, where the diameter of the grids decreases. We analyze under which conditions the integrated mean squared error of $(\hat\eta_{t,h})_{t\in[0,1]^k}$ converges to zero. We start with a general result on generator processes.

\begin{lemma}\label{lem:generator_uniformly_continuous}
Let $(Z_t)_{t\in[0,1]^k}$ be a generator of an SMSP and $\eps_n$, $n\in\N$, be a null sequence. Then
\begin{equation*}
E\left(\sup_{\norm{t-s}\leq \eps_n}\abs{Z_t-Z_s}\right)\to_{n\to\infty}0,
\end{equation*}
where $\norm\cdot$ is an arbitrary norm on $\R^k$.
\end{lemma}

\begin{proof}
The paths of $(Z_t)_{t\in[0,1]^k}$ are continuous, so they are also uniformly continuous. Therefore, $\sup_{\norm{t-s}\leq \eps_n}\abs{Z_t-Z_s}\to_{n\to\infty}0$. Furthermore,
\[
\sup_{\norm{t-s}\leq \eps_n}\abs{Z_t-Z_s}\leq 2\sup_{t\in[0,1]^k}Z_t
\]
with $E\left(\sup_{t\in[0,1]^k}Z_t\right)<\infty$ due to property \eqref{eq:condition_generator} of a generator. The assertion now follows from the dominated convergence theorem.
\end{proof}

Let $\mathcal G_n:=\set{s_{1,n},\dotsc,s_{d(n),n}}$, $n\in\N$, be a set of distinct points in $[0,1]^k$ with the property
\[
\forall n\in\N~\forall t\in[0,1]^k~ \exists s_{i,n}\in\mathcal G_n:~\norm{t-s_{i,n}}\leq\eps_n,
\]
where $\eps_n\to_{n\to\infty}0$. Define, for instance, $\mathcal G_n$ in such a way that
\[
\eps_n:=\max_{i=1,\dotsc,d}\sup_{s,t\in N(s_{i,n})}\norm{s-t}\to_{n\to\infty}0,
\]
with $N(s_{i,n})$ as defined in \eqref{eq:set_closest_points}.  Clearly, $d:=d(n)\to_{n\to\infty}\infty$. Denote by $\norm\cdot_{D^{(n)}_{s_1,\dotsc,s_d}}$ the $D$-norm pertaining to $\eta_{s_{1,n}},\dotsc,\eta_{s_{d,n}}$. Let further $\hat{\bm\eta}_n=(\hat \eta_{t,n})_{t\in[0,1]^k}$ be the kernel-based discretized version of $\bm\eta$ with grid $\mathcal G_n$, that is,
\begin{equation*}
\hat\eta_{t,n}=\max_{i=1,\dotsc,d}\frac{\eta_{s_{i,n}}}{g_{i,n}(t)},\qquad t\in[0,1]^k,
\end{equation*}
where for $i=1,\dotsc,d$
\[
g_{i,n}(t)=\frac{K(\norm{t-s_{i,n}}/h_n)}{\norm{(K(\norm{t-s_{1,n}}/h_n),\dots,K(\norm{t-s_{d,n}}/h_n))}_{D^{(n)}_{s_1,\dotsc,s_d}}},\quad t\in[0,1]^k,
\]
$K:[0,\infty)\to[0,1]$ is the continuous and strictly decreasing kernel function satisfying condition \eqref{eq:condition_on_kernel} and $h_n$, $n\in\N$, is some positive sequence. We have already seen in Lemma \ref{lem:convergence_to_kronecker} that $g_{i,n}(t)\in[0,1]$, $t\in[0,1]^k$, $n\in\N$. Furthermore we have the following result.

\begin{lemma}\label{lem:weight_functions_converge_to_one}
Choose $t\in[0,1]^k$. There is a sequence $i(n)$, $n\in\N$, such that $t\in\bigcap_{n\in\N}N(s_{i(n),n})$. Define $g_{i(n),n}$ and $\eps_n$ as above, $n\in\N$. Then
\[
\lim_{n\to\infty}g_{i(n),n}(t)=1,
\]
if $\eps_n\to_{n\to\infty}0$, $h_n\to_{n\to\infty}0$, $\eps_n/h_n\to_{n\to\infty}\infty$.
\end{lemma}

\begin{proof}
Let $t\in[0,1]^k$ and choose a sequence $i(n)$, $n\in\N$, as above. Put for simplicity $s_{i(n),n}=:s_{i,n}$ and $g_{i(n),n}=:g_{i,n}$. We have
\begin{align*}
1\geq g_{i,n}(t)&=\frac{K\left(\norm{t-s_{i,n}}/h_n\right)}{E\left(\max_{j=1,\dotsc,d}K\left(\norm{t-s_{j,n}}/h_n\right)Z_{s_{j,n}}\right)}\\
&\geq \Bigg(\frac{E\left(\max_{j:\norm{s_{j,n}-t}\geq2\eps_n}K\left(\norm{t-s_{j,n}}/h_n\right)Z_{s_{j,n}}\right)}{K\left(\norm{t-s_{i,n}}/h_n\right)}\\
&\qquad+\frac{E\left(\max_{j:\norm{s_{j,n}-t}<2\eps_n}K\left(\norm{t-s_{j,n}}/h_n\right)Z_{s_{j,n}}\right)}{K\left(\norm{t-s_{i,n}}/h_n\right)}\Bigg)^{-1}\\
&=:(A_{i,n}(t)+B_{i,n}(t))^{-1}.
\end{align*}
From $t\in N(s_{i,n})$ we conclude $\norm{t-s_{i,n}}\leq\eps_n$. Hence, we have due to \eqref{eq:condition_generator} and the properties of the kernel function $K$
\[
0\leq A_{i,n}(t)\leq\frac{K(2\eps_n/h_n)}{K(\eps_n/h_n)}E\left(\sup_{t\in[0,1]^k}Z_t\right)\to_{n\to\infty}0,
\]
since $\eps_n/h_n\to_{n\to\infty}\infty$ by assumption. Furthermore, $t\in N(s_{i,n})$ and the fact that $K$ is decreasing implies
\[
\max_{j:\norm{s_{j,n}-t}<2\eps_n}K\left(\norm{t-s_{j,n}}/h_n\right)=K\left(\norm{t-s_{i,n}}/h_n\right).
\]
Thus,
\begin{align*}
1\leq B_{i,n}(t)&=\frac{1}{K\left(\norm{t-s_{i,n}}/h_n\right)}\bigg(E\bigg(\max_{j:\norm{s_{j,n}-t}<2\eps_n}K\left(\norm{t-s_{j,n}}/h_n\right)Z_{s_{j,n}}\\
&\hspace*{3cm}-\max_{j:\norm{s_{j,n}-t}<2\eps_n}K\left(\norm{t-s_{j,n}}/h_n\right)Z_{s_{i,n}}\bigg)\bigg)+1\\
&\leq\frac{E\left(\max_{j:\norm{s_{j,n}-t}<2\eps_n}K\left(\norm{t-s_{j,n}}/h_n\right)\abs{Z_{s_{j,n}}-Z_{s_{i,n}}}\right)}{K\left(\norm{t-s_{i,n}}/h_n\right)}+1\\
&\leq E\left(\max_{j:\norm{s_{j,n}-t}<2\eps_n}\abs{Z_{s_{j,n}}-Z_{s_{i,n}}}\right)+1\\
&\leq E\left(\sup_{\norm{r-s}<3\eps_n}\abs{Z_{r}-Z_{s}}\right)+1\\
&\to_{n\to\infty}1,
\end{align*}
because of Lemma \ref{lem:generator_uniformly_continuous}. Note that $\norm{s_{j,n}-t}<2\eps_n$ and $t\in N(s_{i,n})$ imply \linebreak $\norm{s_{j,n}-s_{i,n}}<3\eps_n$.
\end{proof}

We have now gathered the tools to prove convergence of the mean squared error to zero.

\begin{theorem}\label{the:mse_kernel_model}
Define $\hat{\bm\eta}_n$ and $\eps_n$ as above, $n\in\N$. Then for every $t\in[0,1]^k$
\[
\MSE\left(\hat\eta_{t,n}\right)\to_{n\to\infty}0,
\]
and
\[
\IMSE\left(\hat\eta_{t,n}\right):=\int_{[0,1]^k}\MSE\left(\hat\eta_{t,n}\right)~dt\to_{n\to\infty}0,
\]
if $\eps_n\to_{n\to\infty}0$, $h_n\to_{n\to\infty}0$, $\eps_n/h_n\to_{n\to\infty}\infty$.
\end{theorem}

\begin{proof}
Denote by
\[
\hat Z_{t,n}=\max_{j=1,\dotsc,d}\left(g_{j,n}( t)Z_{ s_{j,n}}\right),\qquad t\in[0,1]^k,
\]
the generator of $\hat{\bm \eta}_n$. Choose $t\in[0,1]^k$ and a sequence $i:=i(n)$, $n\in\N$, such that $t\in\bigcap_{n\in\N}N\left(s_{i,n}\right)$. We have by Lemma \ref{lem:mse_inequality}, Lemma \ref{lem:weight_functions_converge_to_one} and the continuity of $\bfZ$
\begin{align*}
\MSE\left(\hat\eta_{t,n}\right)&\leq 6E\left(\abs{Z_t-\hat Z_{t,n}}\right)\nonumber\\
&\leq 6E\left(\abs{Z_{t}-Z_{s_{i,n}}}\right)+6E\left(\abs{Z_{s_{i,n}}-g_{i,n}(t)Z_{s_{i,n}}}\right)\nonumber\\
&\quad+6E\left(\abs{g_{i,n}(t)Z_{s_{i,n}}-\hat Z_{t,n}}\right)\nonumber\\
&= 6E\left(\abs{Z_{t}-Z_{s_{i,n}}}\right)+12\left(1-g_{i,n}(t)\right)\nonumber\\
&\to_{n\to\infty}0\label{eqn:bound_for_mse};
\end{align*}
recall that $g_{i,n}(t)Z_{s_{i,n}}\le \hat Z_{t,n}$.

Next we establish convergence of the integrated mean squared error. The sets $N(s_{i,n})$, as defined in \eqref{eq:set_closest_points}, are typically not disjoint, but the intersections $N(s_{i,n})\cap N(s_{j,n})$, $i\neq j$, have Lebesgue measure zero on $[0,1]^k$. Clearly, $\bigcup_{i=1}^dN(s_{i,n})=[0,1]^k$. Therefore, applying Lemma \ref{lem:mse_inequality} yields
\begin{align*}
\IMSE\left(\hat\eta_{t,n}\right)&=\sum_{i=1}^d\int_{N(s_{i,n})}\MSE\left(\hat\eta_{t,n}\right)~dt\\
&\leq6 \sum_{i=1}^d\int_{N(s_{i,n})}E\left(\abs{Z_t-\hat Z_{t,n}}\right)~dt\\
&\leq6\bigg(\sum_{i=1}^d\int_{N(s_{i,n})}E\left(\abs{Z_t-Z_{s_i,n}}\right)~dt\\
&\qquad\qquad+\sum_{i=1}^d\int_{N(s_{i,n})}\abs{1-g_{i,n}(t)}E\left(Z_{s_i,n}\right)~dt\\
&\qquad\qquad+\sum_{i=1}^d\int_{N(s_{i,n})}E\left(\abs{g_{i,n}(t)Z_{s_i,n}-\hat Z_{t,n}}\right)~dt\bigg)\\
&=:6\left(S_{1,n}+S_{2,n}+S_{3,n}\right)
\end{align*}
due to Lemma \ref{lem:mse_inequality}. From Lemma \ref{lem:generator_uniformly_continuous} we conclude
\begin{align*}
S_{1,n}&=\sum_{i=1}^d\int_{N(s_{i,n})}E\left(\abs{Z_t-Z_{s_{i,n}}}\right)~dt\\
&\leq\sum_{i=1}^d\int_{N(s_{i,n})}E\left(\sup_{\norm{r-s}\leq\eps_n}\abs{Z_r-Z_s}\right)~dt\\
&=\int_{[0,1]^k}E\left(\sup_{\norm{r-s}\leq\eps_n}\abs{Z_r-Z_s}\right)~dt\\
&=E\left(\sup_{\norm{r-s}\leq\eps_n}\abs{Z_r-Z_s}\right)\\
&\to_{n\to\infty}0.
\end{align*}
Define
\[
A_n:=\frac{K(2\eps_n/h_n)}{K(\eps_n/h_n)}E\left(\sup_{t\in[0,1]^k}Z_t\right),\quad B_n:=E\left(\sup_{\norm{r-s}<3\eps_n}\abs{Z_{r}-Z_{s}}\right)+1.
\]
As we have seen in the proof of Lemma \ref{lem:weight_functions_converge_to_one}, we have for $t\in N(s_{i,n})$
\[
1\geq g_{i,n}(t)\geq (A_n+B_n)^{-1}\to1,
\]
and therefore
\begin{align*}
S_{2,n}&=\sum_{i=1}^d\int_{N(s_{i,n})}(1-g_{i,n}(t))~dt\\
&\leq\sum_{i=1}^d\int_{N(s_{i,n})}1-(A_n+B_n)^{-1}~dt\\
&=\int_{[0,1]^k}1-(A_n+B_n)^{-1}~dt\\
&=1-(A_n+B_n)^{-1}\\
&\to_{n\to\infty}0.
\end{align*}
Lastly, we have by the same argument as above
\begin{equation*}
S_{3,n}=\sum_{i=1}^d\int_{N(s_{i,n})}E\left(\hat Z_{t,n}-g_{i,n}(t)Z_{s_i,n}\right)~dt=S_{2,n}\to_{n\to\infty}0,
\end{equation*}
which completes the proof.
\end{proof}

\begin{rem}
\upshape Given a grid $s_1,\dots,s_{d(n)}$ with pertaining $\varepsilon_n$, the bandwidth $h_n:=\varepsilon_n^2$ would, for example, satisfy the required growth conditions entailing convergence of MSE and IMSE to zero. But, it would clearly be desirable to provide some details on how to choose the bandwidth in an optimal way, which is, for example, statistical folklore in kernel density estimation. In our setup, however, this is an open problem, which requires future work.
\end{rem}

\section{Discretized versions of copula processes}

Next we transfer the model we have established in Section \ref{sec:model} to copula processes that are in a sense close to max-stable processes. A \emph{copula process} $\bm U=(U_t)_{t\in[0,1]^k}$ is a stochastic process with continuous sample paths, such that each rv $U_t$ is uniformly distributed on the interval $[0,1]$. We say that $\bm U$ is in the \emph{functional domain of attraction} of an SMSP $\bm\eta=(\eta_t)_{t\in[0,1]^k}$, if
\begin{equation}\label{eq:fdoa}
\lim_{n\to\infty}P\left(n\left(\bm U-1\right)\leq f\right)^n=P\left(\bm\eta\leq f\right)=\exp\left(-\norm f_D\right),\qquad f\in\bar E^-\left([0,1]^k\right).
\end{equation}
Define for any $t\in[0,1]^k$ and $n\in \N$
\[
Y_t^{(n)}:=n\left(\max_{i=1,\dotsc,n}U_t^{(i)}-1\right),
\]
with $\bm U^{(1)},\bm U^{(2)},\dotsc$ being independent copies of $\bm U$. Now choose again pairwise different points $s_1,\dotsc,s_d\in[0,1]^k$ and functions $g_1,\dotsc,g_d\in\bar C^+\left([0,1]^k\right)$ with the properties \eqref{eq:norming_functions_standardization} and \eqref{eq:norming_functions_interpolation}. Condition \eqref{eq:fdoa} implies weak convergence of the finitedimensional distributions of $\bm Y^{(n)}=(Y_t^{(n)})_{t\in[0,1]^k}$, i.\,e.
\[
\left(Y_{s_1}^{(n)},\dotsc,Y_{s_d}^{(n)}\right)\to_{\mathcal D}\left(\eta_{s_1},\dotsc,\eta_{s_d}\right),
\]
where '$\to_{\mathcal D}$' denotes convergence in distribution. Just like before, we can define the \emph{discretized version} $\hat{\bm Y}^{(n)}=(\hat Y^{(n)}_t)_{t\in[0,1]^k}$ of $\bm Y^{(n)}$ with grid $\{s_1,\dotsc,s_d\}$ and weight functions $g_1,\dotsc,g_d$ to be
\[
\hat Y^{(n)}_t:=\max_{i=1,\dotsc,d}\frac{Y_{s_i}^{(n)}}{g_i(t)},\qquad t\in[0,1]^k.
\]
Elementary calculations show that \eqref{eq:fdoa} implies
\[
\lim_{n\to\infty}P\left(\hat{\bm Y}^{(n)}\leq f\right)=P\left(\hat{\bm\eta}\leq f\right),\qquad f\in \bar E^-\left([0,1]^k\right),
\]
where $\hat{\bm\eta}$ is the discretized version of $\bm\eta$ as defined in \eqref{eq:generalized_max_linear_model}. Also, it is not difficult to see that for each $t\in[0,1]^k$,
\begin{equation*}\label{eq:doa_bivariate_smsrv}
\left(Y_t^{(n)},\hat Y_t^{(n)}\right)\to_{\mathcal D}(\eta_t,\hat\eta_t)
\end{equation*}
where $(\eta_t,\hat\eta_t)$ is the standard max-stable rv from Lemma \ref{lem:eta_hateta_sms}. Now applying the continuous mapping theorem, we obtain
\begin{equation*}
\left(Y_t^{(n)}-\hat Y_t^{(n)}\right)^2\to_{\mathcal D}(\eta_t-\hat\eta_t)^2.
\end{equation*}
It remains to prove uniform integrability of the sequence on the left hand side in order to obtain the next result.

\begin{prop}
Let $t\in[0,1]^k$. Then
\[
\MSE\left(\hat Y_{t}^{(n)}\right)=E\left(\left(Y_t^{(n)}-\hat Y_t^{(n)}\right)^2\right)\to_{n\to\infty}\MSE\left(\hat\eta_t\right).
\]
\end{prop}

\begin{proof}
Fix $t\in[0,1]^k$. It remains to show that the sequence $X_t^{(n)}:=\left(Y_t^{(n)}-\hat Y_t^{(n)}\right)^2$ is uniformly integrable. A sufficient condition for uniform integrability is
\[
\sup_{n\in\N}E\left(\left(X_t^{(n)}\right)^2\right)<\infty,
\]
see \citet[Section 3]{billi99}. Clearly, for every $n\in\N$,
\[
E\left(\left(X_t^{(n)}\right)^2\right)\leq E\left(\left(Y_t^{(n)}\right)^4\right)+E\left(\left(\hat Y_t^{(n)}\right)^4\right).
\]
It is easy to verify that the rv $Y_t^{(n)}$ has the density $(1+x/n)^{n-1}$ on $[-n,0]$. Therefore,
\begin{align*}
E\left(\left(Y_t^{(n)}\right)^4\right)=\int_{-n}^0x^4\left(1+\frac xn\right)^{n-1}~dx=\frac{24n^5(n-1)!}{(n+4)!}\leq 24.
\end{align*}
Moreover, putting $c:=\min_{i=1,\dotsc,d}g_i(t)>0$,
\[
\abs{\hat Y_t^{(n)}}=\min_{i=1,\dotsc,d}\frac{\abs{Y_{s_i}^{(n)}}}{g_i(t)}\leq\frac{\abs{Y_{s_1}^{(n)}}}{c},
\]
and hence
\[
E\left(\left(\hat Y_t^{(n)}\right)^4\right)\leq\frac{24}{c^4},
\]
which completes the proof.
\end{proof}

\section*{Acknowledgment}
The authors are grateful to two anonymous reviewers for their careful reading of the manuscript. The paper has benefitted a lot from their critical remarks.

\end{document}